%% file: main.tex
\pdfoutput=1
\documentclass[a4paper,reqno]{amsart}
\usepackage[margin=3cm]{geometry}
\input{core}
\bibliography{references}

\usepackage{centernot}

\newcommand{\Kl}{\b{Kl}}
\renewcommand{\Im}{\b{Im}}
\newcommand{\idem}{_{\tx{idem}}}

\title{Idempotence for relative monads}
\author{Nathanael Arkor}
\address{Department of Software Science, Tallinn University of Technology, Estonia}
\author{Andrew Slattery}
\address{Department of Computer Science and Technology, University of Cambridge, United Kingdom}
\thanks{N.\ Arkor was supported by a departmental postdoctoral grant from the Department of Software Science at Tallinn University of Technology.}
\subjclass{18C15, 18C20}
\thanks{A.\ Slattery was funded in part by the United States Air Force Office of Scientific Research under grant FA9550-23-1-0728 (New Spaces for Denotational Semantics; Dr. Tristan Nguyen, Program Manager). Views and opinions expressed are however those of the author only and do not necessarily reflect those of AFOSR.}
\subjclass{18C15, 18C20}
\datetoday

\begin{document}

\begin{abstract}
    We study the concept of idempotence for relative monads, which exhibits several subtleties not present for non-relative monads. In particular, there is a bifurcation of notions of idempotence in the relative setting, which are indistinguishable for idempotent monads. As a special case, we obtain several characterisations of idempotence for monads in extension form.
\end{abstract}

\maketitle

Idempotent monads were introduced by \textcite{appelgate1969categories} in the study of codensity monads; a monad~$(t, \mu, \eta)$ is \emph{idempotent} if its multiplication $\mu \colon tt \tto t$ is invertible. Relative monads were introduced by \textcite{lavendhomme1974variations} and developed further by \textcite{altenkirch2010monads} (with precursors in \cite{walters1969alternative,walters1970categorical,diers1975jmonades,altenkirch1999monadic}) as a generalisation of monads from structured endofunctors to arbitrary functors. Here, we study the appropriate generalisation of idempotence from monads to relative monads. Notions of idempotent relative monad have appeared twice before in the literature: in \citeauthor{diers1975jmonades}'s study of \emph{idempotent $j$-monads}~\cite[\S4]{diers1975jmonades}; and, in the two-dimensional setting, in \citeauthor{fiore2018relative}'s study of \emph{lax-idempotent relative pseudomonads}~\cite[\S5]{fiore2018relative}. Unfortunately, these notions fail to satisfy the evident generalisations of known properties of idempotent monads. We show that, by strengthening the notion of idempotence for relative monads -- and consequently distinguishing between \emph{non-algebraic} and \emph{algebraic} idempotence -- we can recover many of these desirable properties. Our work is in fact motivated by the two-dimensional setting: in joint work with Philip Saville, we have studied an analogous algebraic notion of lax-idempotence for relative pseudomonads~\cite[\S5]{arkor2025bicategories}. However, as the subtleties surrounding idempotence arise even in the one-dimensional setting, it seemed valuable to provide a self-contained account.

This paper may be seen as a particularly simple case study of a general phenomenon in the theory of relative monads. In the theory of (non-relative) monads, it is typically sensible to define a property or structure of a monad and then observe that it automatically interacts well with the algebras for the monad. Intuitively, this is the case because the structure of an algebra for a monad is determined by the structure of free algebras. However, this is generally not true for relative monads: in contrast, one must typically define properties or structure both for the relative monad \emph{and} for its algebras. We expect similar observations to play an important role, for instance, in the theory of strong relative monads~\cite{slattery2023pseudocommutativity}, where one may define a notion of \emph{strong algebra} in addition to the existing notion of \emph{strong relative monad}.

\subsection*{Setting}

For expository purposes, we shall work throughout in the context of monads relative to functors between ordinary categories. The appropriate generalisations to enriched relative monads and, more generally, to relative monads in a virtual equipment~\cite{arkor2024formal} are evident, and are left to the reader; there are no essential differences to the theory presented here. Our notation primarily follows that of \textcite{arkor2024formal}, though we do not assume prior familiarity.

\subsection*{Acknowledgements}

The authors thank Philip Saville and Dylan McDermott for helpful comments.

\section{Idempotent monads and idempotent relative monads}

Throughout, we denote by $j \colon A \to E$ a fixed functor (called the \emph{root}). We begin by recalling the concept of a relative monad, before introducing the concept of idempotence therefor.

\begin{definition}[{\cite[Definition~2.1]{altenkirch2015monads}}]
    A \emph{$j$-relative monad} $T$ comprises the following data.
    \begin{enumerate}
        \item For each object $a \in A$, an object $ta \in E$ and a morphism $\eta_a \colon ja \to ta$.
        \item For each morphism $f \colon ja \to tb$ in $E$, a morphism $f^\dag \colon ta \to tb$.
    \end{enumerate}
    The following laws must be satisfied.
    \begin{enumerate}[resume]
        \item $f^\dag \eta_a = f$, for each morphism $f \colon ja \to tb$ in $E$.
        \item $(\eta_a)^\dag = 1_{ta}$, for each object $a \in A$.
        \item $(g^\dag f)^\dag = g^\dag f^\dag$, for each pair of morphisms $f \colon ja \to tb$ and $g \colon jb \to tc$.
        \qedhere
    \end{enumerate}
\end{definition}

It follows that $t$ extends uniquely to a functor $t \colon A \to E$ for which the transformations
\[\eta \colon j \tto t \colon A \to E\]
\[\dag \colon E(j, t) \tto E(t, t) \colon A\op \times A \to \Set\]
are natural~\cite[Theorem~8.12]{arkor2024formal}. We call $\eta$ the \emph{unit} and $\dag$ the \emph{extension operator}.

\begin{example}
    There is a bijection between monads on a category $E$ and monads relative to the identity functor $1_E$ on $E$~\cite[Corollary~4.20]{arkor2024formal}. Given a monad $(t, \mu, \eta)$, we define $f^\dag \defeq \mu_b \c tf$; conversely, given a $1_E$-relative monad, we define $\mu_e \defeq (1_{te})^\dag$.
\end{example}

\begin{definition}
    \label{idempotent-relative-monad}
    A $j$-relative monad $T$ is \emph{idempotent} if its extension operator $\dag \colon E(j, t) \tto E(t, t)$ is invertible.
\end{definition}

\begin{example}
    \label{composing}
    Let $j' \colon E \to I$ be a functor and $T$ be a $j'$-relative monad. Then we may form a $j'j$-relative monad $Tj$ by precomposing an arbitrary functor $j \colon A \to E$~\cite[Proposition~5.36]{arkor2024formal}: $Tj$ is idempotent if $T$ is idempotent.

    In particular, each functor $j \colon A \to E$, viewed as a trivial $j$-relative monad, is idempotent, since it is given by precomposing the identity monad on $E$ by $j$.
\end{example}

\begin{example}
    Let $j \colon A \to E$ be a functor and $T$ be a $j$-relative monad. We may form a $j'j$-relative monad $j'T$ by postcomposing a \ff{} functor $j' \colon E \to I$~\cite[Example~5.38]{arkor2024formal}. $j'T$ is idempotent if and only if $T$ is idempotent.
\end{example}

\begin{remark}
    We direct readers interested in \emph{strong} relative monads to \cite[\S2.2]{slattery2024commutativity}, noting that every idempotent strong relative monad is automatically commutative~\cite[Theorem~2.12]{slattery2024commutativity}.
\end{remark}

There are several equivalent formulations of idempotence for relative monads.

\begin{lemma}
    \label{idempotent-relative-monad-characterisations}
    The following are equivalent for a $j$-relative monad $(t, \dag, \eta)$.
    \begin{enumerate}
        \item $\dag$ is invertible.
        \item $\dag$ is an epimorphism.
        \item $\dag$ is right-inverse to $E(\eta, t)$.
        \item $\dag$ is inverse to $E(\eta, t)$.
        \item For each object $b \in A$, the object $tb$ is $\eta$-orthogonal, \ie{} every morphism $f \colon ja \to tb$ admits a unique extension along~$\eta_a$.
        \[\begin{tikzcd}
        	ta \\
        	ja & tb
        	\arrow["{\exists!}", dashed, from=1-1, to=2-2]
        	\arrow["{\eta_a}", from=2-1, to=1-1]
        	\arrow["f"', from=2-1, to=2-2]
        \end{tikzcd}\]
        \item[(1')] $E(\eta, t)$ is invertible.
        \item[(2')] $E(\eta, t)$ is a monomorphism.
        \item[(3')] $E(\eta, t)$ is left-inverse to $\dag$.
    \end{enumerate}
\end{lemma}

We defer the easy proof to \cref{idempotent-algebra-characterisations}, where we shall establish a more general result. It follows that our definition of idempotent relative monad coincides\footnotemark{} with that of \textcite[\S4]{diers1975jmonades}, who defined idempotence in terms of condition (2'). Explicitly, condition (4) states that a relative monad is idempotent if and only if, for every morphism $g \colon ta \to tb$ in $E$, we have $(g\eta_a)^\dag = g$. Consequently, it also follows that our definition of idempotent relative monad coincides with the definition of lax-idempotent relative pseudomonad given by \textcite[Definition~5.1]{fiore2018relative} (and subsequently simplified in \cite[Proposition~5.4]{arkor2025bicategories}), viewing a relative monad as a pseudomonad relative to a 2-functor between locally discrete 2-categories~\cite[Theorem~5.3(i $\iff$ ii)]{fiore2018relative}.
\footnotetext{At least after noting that \citeauthor{diers1975jmonades}'s \emph{$j$-monads} are equivalent to $j$-relative monads~\cite[Example~8.14]{arkor2024formal}.}

We note in passing that an idempotent relative monad admits a simpler presentation than an arbitrary relative monad due to uniqueness of the extensions.

\begin{proposition}
    \label{presentation-of-idempotent-relative-monad}
    An idempotent $j$-relative monad is uniquely determined by the following data.
    \begin{enumerate}
        \item For each object $a \in A$, an object $ta \in E$ and a morphism $\eta_a \colon ja \to ta$.
        \item For each morphism $f \colon ja \to tb$ in $E$, a unique morphism $f^\dag \colon ta \to tb$ such that $f^\dag \eta_a = f$.
    \end{enumerate}
\end{proposition}

\begin{proof}
    By \cref{idempotent-relative-monad-characterisations}, every idempotent $j$-relative monad satisfies the second condition. Conversely, the remaining unit law and the associativity law for a relative monad follow from uniqueness of the extensions.
\end{proof}

To justify \cref{idempotent-relative-monad}, we shall show that it coincides with the notion of idempotence for non-relative monads when $j$ is the identity functor (\cf{} the analogous statement for lax-idempotent relative pseudomonads~\cite[Remark~5.5]{fiore2018relative}). First, it will be useful to recall several equivalent characterisations of idempotent monads.

\begin{lemma}
    \label{idempotent-monad-characterisations}
    The following are equivalent for a monad $T = (t, \mu, \eta)$ on a category $E$.
    \begin{enumerate}
        \item $\mu \colon tt \tto t$ is invertible.
        \item \label{t-eta-condition} $t\eta = \eta t \colon t \tto tt$.
        \item \label{algebras-are-fixed-points} Every $T$-algebra $(e, \epsilon \colon te \to e)$ is a fixed point, \ie{} $\epsilon$ is invertible.
        \item \label{free-algebras-are-fixed-points} Every free $T$-algebra is a fixed point.
        \item \label{forgetful-functor-is-ff} The forgetful functor $\Alg(T) \to E$ is full, hence \ff{}.
    \end{enumerate}
\end{lemma}

\begin{proof}
    (1) $\iff$ (2) $\iff$ (3) $\iff$ (5) are \cite[Propositions~6.1 -- 6.3]{appelgate1969categories}. (1) $\iff$ (4) is trivial, since the algebra structure of a free algebra is given by $\mu$.
\end{proof}

\begin{proposition}
    \label{monad-idempotent-iff-relative-monad-idempotent}
    A monad $T = (t, \mu, \eta)$ on $E$ is idempotent if and only if it is idempotent as a $1_E$-relative monad.
\end{proposition}

\begin{proof}
    Suppose that $T$ is idempotent as a non-relative monad, and let $(e, \epsilon)$ be an algebra. Let $f \colon tx \to e$ be a morphism. Then $(f \eta_x)^\dag = \epsilon \c t(f \eta_x) = \epsilon \c tf \c t\eta_x = \epsilon \c tf \c \eta_{tx} = \epsilon \c \eta_{ty} \c f = f$ using the definition of $\dag$ in terms of $\epsilon$, functoriality of $t$, idempotence, naturality of $\eta$, and the unit law. Hence $T$ is idempotent as a $1_E$-relative monad.

    Conversely, suppose that $T$ is idempotent as a $1_E$-relative monad. For all objects $e$, we have $t \eta_e = (\eta_{te} \eta_e)^\dag = \eta_{te}$ using first the definition of the functoriality of $t$ in terms of $\dag$ and then idempotence. Thus, by \cref{idempotent-monad-characterisations}, $T$ is idempotent as a non-relative monad.
\end{proof}

One might then hope that each of the equivalent characterisations of \cref{idempotent-monad-characterisations} extend to analogous characterisations of idempotent relative monads (with the exception of \cref{t-eta-condition}, which fundamentally involves iteration of the functor $t$, and so does not have an evident generalisation to relative monads). Unfortunately, this is not the case: we shall give a counterexample in \cref{idempotence-counterexample}. This motivates the study of a stronger notion of idempotence for relative monads.

\section{Algebraically idempotent relative monads}

We begin by recalling the concept of an algebra for a relative monad, before introducing the concept of idempotence therefor.

\begin{definition}[{\cite[Definition~2.11]{altenkirch2015monads}}]
    An \emph{algebra} for a $j$-relative monad $T$ (or simply \emph{$T$-algebra} comprises the following data.
    \begin{enumerate}
        \item An object $e \in E$.
        \item For each morphism $f \colon ja \to e$ in $E$, a morphism $f^\rtimes \colon ta \to e$.
    \end{enumerate}
    The following laws must be satisfied.
    \begin{enumerate}[resume]
        \item $f^\rtimes \eta_a = f$, for each morphism $f \colon ja \to e$ in $E$.
        \item $(g^\rtimes f)^\rtimes = g^\rtimes f^\dag$, for each pair of morphisms $f \colon ja \to tb$ and $g \colon jb \to e$.
    \end{enumerate}
    A \emph{morphism} of $T$-algebras from $(e, \rtimes)$ to $(e', \rtimes')$ is a morphism $\epsilon \colon e \to e'$ such that $\epsilon \c f^\rtimes = (\epsilon f)^\rtimes$, for each morphism $f \colon ja \to e$ in $E$. $T$-algebras and their morphisms form a category $\Alg(T)$.
\end{definition}

\begin{example}
    For every $j$-relative monad $T = (t, \dag, \eta)$ and object $a \in A$, the pair $(ta, \dag)$ forms a $T$-algebra, called the \emph{free $T$-algebra}.
\end{example}

Several of the equivalent characterisations of idempotent monads involve properties of their algebras. This motivates the following definition.

\begin{definition}
    Let $T$ be a $j$-relative monad. A $T$-algebra $(e, \rtimes)$ is \emph{idempotent} if its extension operator $\rtimes \colon E(j, e) \tto E(t, e)$ is invertible. Denote by $\Alg\idem(T)$ the full subcategory of $\Alg(T)$ spanned by the idempotent algebras.
\end{definition}

As with idempotent relative monads, there are several equivalent formulations of idempotence for algebras.

\begin{lemma}
    \label{idempotent-algebra-characterisations}
    The following are equivalent.
    \begin{enumerate}
        \item $\rtimes$ is invertible.
        \item $\rtimes$ is an epimorphism.
        \item $\rtimes$ is right-inverse to $E(\eta, e)$.
        \item $\rtimes$ is inverse to $E(\eta, e)$.
        \item $e$ is $\eta$-orthogonal, \ie{} every morphism $f \colon ja \to e$ admits a unique extension along~$\eta_a$.
        \[\begin{tikzcd}
        	ta \\
        	ja & e
        	\arrow["{\exists!}", dashed, from=1-1, to=2-2]
        	\arrow["{\eta_a}", from=2-1, to=1-1]
        	\arrow["f"', from=2-1, to=2-2]
        \end{tikzcd}\]
        \item[(1')] $E(\eta, e)$ is invertible.
        \item[(2')] $E(\eta, e)$ is a monomorphism.
        \item[(3')] $E(\eta, e)$ is left-inverse to $\rtimes$.
    \end{enumerate}
\end{lemma}

\begin{proof}
    The first axiom for a $T$-algebra asserts that $E(\eta, e) \c \rtimes = 1$, so that $\rtimes$ is a split monomorphism, from which the equivalence of (1 -- 4) follows. (1' -- 3') follow from (4) by duality. For (4) $\iff$ (5), existence is given by $f^\rtimes$; for uniqueness, suppose there is any morphism $f' \colon ta \to e$ such that $f' \eta_a = f$. Then $f' = (f' \eta_a)^\rtimes = f^\rtimes$.
\end{proof}

Observe that this implies that, if an object $e$ admits an idempotent algebra structure, that algebra structure is unique.

As with idempotent relative monads, idempotent algebras admit a simpler presentation than arbitrary algebras.

\begin{proposition}
    \label{presentation-of-idempotent-algebra}
    Let $T = (t, \dag, \eta)$ be a $j$-relative monad. An idempotent $T$-algebra is uniquely determined by the following data.
    \begin{enumerate}
        \item An object $e \in E$.
        \item For each morphism $f \colon ja \to e$ in $E$, a unique morphism $f^\rtimes \colon ta \to e$ such that $f^\rtimes \eta_a = f$.
    \end{enumerate}
\end{proposition}

\begin{proof}
    By \cref{idempotent-algebra-characterisations}, every idempotent $j$-relative monad satisfies the second condition. Conversely, the compatibility law for an algebra follow from uniqueness of the extensions.
\end{proof}

Consequently, we may see that an idempotent relative monad satisfies an analogue of \cref{free-algebras-are-fixed-points}.

\begin{lemma}
    \label{idempotent-iff-free-algebras-are-idempotent}
    A $j$-relative monad $T$ is idempotent if and only if every free $T$-algebra is idempotent.
\end{lemma}

\begin{proof}
    Observe that, for each object $a \in A$, the free $T$-algebra on $a$ is equipped with the extension operator $\dag_{{-}, a} \colon E(j, ta) \tto E(t, ta)$. Hence $T$ is idempotent if and only if $\dag$ is invertible if and only if, for all $a \in A$, $\dag_{{-}, a}$ is invertible, if and only if every free $T$-algebra is idempotent.
\end{proof}

We are particularly interested in analogues of \cref{algebras-are-fixed-points,forgetful-functor-is-ff}. This motivates the introduction of the following, stronger notion of idempotence for relative monads.

\begin{definition}
    A $j$-relative monad $T$ is \emph{algebraically idempotent} if every $T$-algebra is idempotent.
\end{definition}

\begin{example}
    Every functor $j \colon A \to E$, viewed as a trivial $j$-relative monad is algebraically idempotent: algebras for trivial relative monads are in bijection with objects of the codomain~\cite[Proposition~6.42]{arkor2024formal}, whose extension operator $\rtimes$ is the identity.
\end{example}

By \cref{idempotent-iff-free-algebras-are-idempotent}, every algebraically idempotent relative monad is idempotent. However, the converse is not true.

\begin{example}[Idempotence $\centernot\implies$ algebraic idempotence]
    \label{idempotence-counterexample}
    Let $E$ be the free category on the graph
    \[\begin{tikzcd}
    	j & t & e
    	\arrow["\eta", from=1-1, to=1-2]
    	\arrow["{f'}"', curve={height=6pt}, from=1-2, to=1-3]
    	\arrow["f", curve={height=-6pt}, from=1-2, to=1-3]
    \end{tikzcd}\]
    modulo $f \eta = f' \eta$. We can view $j$ and $t$ as functors $1 \to E$ from the terminal category. The functor $t$ may then be equipped with a relative monad structure $T$: the unit is $\eta$, and the extension operator sends the morphism $\eta$ to the identity on $t$: this is trivially a bijection because $t$ has no nontrivial endomorphisms. Thus $T$ is idempotent. However, $e$ has two distinct algebra structures: one sends the unique morphism $j \to e$ to $f$, and the other sends the unique morphism to $f'$. Thus $e$ admits a non-idempotent algebra structure.
\end{example}

\begin{remark}
    We expect that idempotence does not imply algebraic idempotence even when the root $j$ is dense, though we do not have a counterexample. (\textcite[\S4]{diers1975jmonades} claims that idempotence is weaker than algebraic idempotence even in the presence of dense roots, but without giving a counterexample.)
\end{remark}

For non-relative monads, the two notions of idempotence coincide, as suggested by \cref{idempotent-monad-characterisations}.

\begin{lemma}
    \label{idempotent-monad-is-algebraically-idempotent}
    Every idempotent (non-relative) monad is algebraically idempotent.
\end{lemma}

\begin{proof}
    Let $(t, \mu, \eta)$ be an idempotent monad and let $(e, \rtimes)$ be an algebra. Let $f \colon tx \to e$ be a morphism. Then $(f \eta_x)^\rtimes = \epsilon \c t(f \eta_x) = \epsilon \c tf \c t\eta_x = \epsilon \c tf \c \eta_{tx} = \epsilon \c \eta_{ty} \c f = f$ using the definition of $\rtimes$ in terms of $\epsilon$, functoriality of $t$, idempotence, naturality of $\eta$, and the unit law.
\end{proof}

\begin{remark}
    It is important to note that, while it is true that a non-relative monad is idempotent if and only if its multiplication $\mu$ is invertible if and only if its extension operator $\dag$ is invertible, the analogous statement is not generally true for algebras. That is, it is not generally true that an algebra is a fixed point if and only if it is idempotent.
    
    For instance, consider the monad $\ph + 1$ on the category of sets. An algebra for $\ph + 1$ is a pointed set, \ie{} a set equipped with a specified element. A pointed set $(Y, y)$ is idempotent as an algebra if and only if, for any set $X$ and function $[f, y'] \colon X + 1 \to Y$, we have $[f, y'] = [f, y]$, hence if and only if $Y$ is a singleton. Conversely, $\ph + 1$ has no fixed points, since the coprojection $\copi_1 \colon Y \to Y + 1$ is not surjective, hence never an isomorphism. In fact, this shows that it is not even true that a free algebra is idempotent if and only if it is a fixed point, because singletons are free for $\ph + 1$ on $\varnothing$. However, \cref{monad-idempotent-iff-relative-monad-idempotent,idempotent-monad-is-algebraically-idempotent,idempotent-monad-characterisations} implies that if \emph{every} free algebra is idempotent, then idempotence is equivalent to being a fixed point.

    More generally, for a monad $T$ on a category with a terminal object, $1$ is always an idempotent algebra, but is a fixed point if and only if $t$ preserves the terminal object.

    The relationship between idempotent algebras and fixed points is studied in the greater generality of relative pseudomonads in \cite[\S5.2]{arkor2025bicategories}.
\end{remark}

Following the previous remark, there is a bifurcation of concepts in the relative setting: the weaker \emph{idempotence}, and the stronger \emph{algebraic idempotence}. Our thesis is that it is algebraic idempotence that is the appropriate generalisation. We support this perspective by showing in the next section that algebraic idempotence recovers the final equivalent characterisation of idempotent monads in \cref{idempotent-monad-characterisations} (and hence, by \cref{idempotence-counterexample}, that idempotence alone does not). This suggests it may be appropriate to use the term \emph{idempotence} for algebraic idempotence, and refer to the weaker notion as \emph{non-algebraic idempotence} (we have not done so herein purely for consistency with previous literature).

\begin{remark}
    By \cite[Corollary~6.40]{arkor2024formal}, every $j$-relative monad morphism $\tau \colon T \to T'$ induces a functor $\Alg(\tau) \colon \Alg(T') \to \Alg(T)$ over $E$ (moreover, such functors are in bijection with relative monad morphisms). However, note that this functor generally does not restrict to a functor $\Alg\idem(T') \to \Alg\idem(T)$ between categories of idempotent algebras unless $\tau$ is invertible.
\end{remark}

\section{(Algebraically) idempotent relative adjunctions}

Relative adjunctions, which were introduced by \textcite{ulmer1968properties}, are to relative monads what adjunctions are to monads.

\begin{definition}
    A \emph{$j$-relative adjunction} comprises functors $\ell \colon A \to C$ and $r \colon C \to E$ together with a bijection
    \[C(\ell a, c) \iso E(j a, r c)\]
    natural in $a \in A$ and $c \in C$. We denote this situation by $\ell \radj j r$.
\end{definition}

\begin{example}
    \label{algebra-resolution}
    For each $j$-relative monad $T$, the functor $f_T \colon A \to \Alg(T)$ sending each object $a \in A$ to the free $T$-algebra $(ta, \dag)$ is left-adjoint to the forgetful functor $u_T \colon \Alg(T) \to E$, relative to $j$.
\end{example}

\begin{example}
    \label{opalgebra-resolution}
    For every $j$-relative monad $T$, we can form the \emph{Kleisli category} $\Kl(T)$ whose objects are the same as those of $A$, and for which a morphism from $a$ to $b$ is a morphism $ja \to tb$ in $E$. The \ioo{} functor $k_T \colon A \to \Kl(T)$ sending a morphism $f \colon a \to b$ in $A$ to a morphism $\eta_b \c jf \colon ja \to jb \to tb$ is left-adjoint, relative to $j$, to the functor $v_T \colon \Kl(T) \to E$ sending an object $a \in A$ to the object $ta$.
\end{example}

\begin{example}
    Every $j$-relative adjunction induces a $j$-relative monad structure on the composite $r \ell$ by defining the unit as the image of $1_\ell$ under $C(\ell, \ell) \iso E(j, r\ell)$, and the extension operator as $E(j, r\ell) \iso C(\ell, \ell) \tto E(r\ell, r\ell)$~\cite[Theorem~5.24]{arkor2024formal}. We say that a $j$-relative adjunction is a \emph{resolution} of a $j$-relative monad $T$ if it induces $T$ in this way~\cite[Definition~5.25]{arkor2024formal}; a morphism of resolutions is a functor $c \colon C \to C'$ between the apices of the relative adjunctions commuting with the left and right adjoints. In particular, the $j$-relative adjunctions $f_T \radj j u_T$ and $k_T \radj j v_T$ of \cref{algebra-resolution,opalgebra-resolution} are resolutions of $T$ (in fact, they are the initial and terminal resolutions respectively~\cite[Corollary~6.41 \& Corollary~6.51]{arkor2024formal}).
\end{example}

An adjunction is \emph{idempotent} if it induces an idempotent monad~\cite[10]{macdonald1982tower}. We introduce the analogous concepts for relative adjunctions.

\begin{definition}
    A relative adjunction is \emph{(algebraically) idempotent} if it induces an (algebraically) idempotent relative monad.
\end{definition}

We shall not attempt to give a full characterisation of idempotent relative adjunctions analogous to the non-relative setting~\cite[Proposition~2.8]{macdonald1982tower}. However, the following special case will be important.

\begin{definition}[{\cite[\S4]{diers1975jmonades}}]
    A functor $r \colon C \to E$ is \emph{$j$-reflective} (or is a \emph{$j$-reflection}) if it is \ff{} and admits a left $j$-relative adjoint $\ell \colon A \to C$.
\end{definition}

Note in particular that $j$-reflections are \emph{resolute} in the sense of \cite[Definition~6.1.4]{arkor2022monadic}, \ie{} the left $j$-relative adjoint $\ell$ is also left-adjoint to $r$ relative to the composite $r\ell$.

\begin{proposition}[{\cf{}~\cite[\S4]{diers1975jmonades}}]
    \label{j-reflections-are-idempotent}
    Every $j$-reflection is idempotent.
\end{proposition}

\begin{proof}
    Suppose $\ell$ is left $j$-adjoint to $r$. The operator of the induced $j$-monad is given by $E(j, r\ell) \iso C(\ell, \ell) \iso E(r\ell, r\ell)$.
\end{proof}

The converse does not hold: that is, not every idempotent relative adjunction is a relative reflection (indeed, this is not true even for non-relative monads). Nonetheless, relative reflections appear to be an important notion for (algebraically) idempotent relative monads, as we shall show.

\subsection{Free algebras and reflective resolutions}

We fix a functor $j \colon A \to E$ and a $j$-relative monad $T = (t, \dag, \eta)$ for the remainder. Observe that the underlying functor $t \colon A \to E$ factors as
\[A \to \Kl(T) \to \Im(t) \to E\]
where $A \to \Im(t) \to E$ is the full image factorisation of $t$ into an \ioo{} functor followed by a \ff{} functor. The \ioo{} functor $\Kl(T) \to \Im(t)$ sends each morphism $f \colon ja \to tb$ to $f^\dag \colon ta \to tb$.

\begin{proposition}
    \label{Kl(T)-is-im(t)}
    Let $T = (t, \dag, \eta)$ be an relative monad. $T$ is idempotent if and only if the canonical functor $\Kl(T) \to \Im(t)$ is \ff{}, hence an isomorphism.
\end{proposition}

\begin{proof}
    The canonical \ioo{} functor is given on morphisms by $\dag$. Hence $\dag$ is invertible if and only if the canonical functor is \ff{}.
\end{proof}

\begin{corollary}
    \label{T-idempotent-iff-Kl(T)-is-j-reflective}
    The following are equivalent.
    \begin{enumerate}
        \item $T$ is idempotent.
        \item The Kleisli category $\Kl(T)$ forms a $j$-reflective resolution $k_T \radj j v_T$ of $T$.
        \item $T$ admits a $j$-reflective resolution.
    \end{enumerate}
\end{corollary}

\begin{proof}
    (1) $\implies$ (2). If $T$ is idempotent, then the functor $v_T \colon \Kl(T) \to \Im(t)$ is \ff{} by \cref{Kl(T)-is-im(t)}, while $\Im(t) \to E$ is \ff{} by definition, so the composite is also \ff.
    
    (2) $\implies$ (3). Trivial.

    (3) $\implies$ (1). Immediate from \cref{j-reflections-are-idempotent}.
\end{proof}

\begin{corollary}[{\cf{}~\cite[\S4]{diers1975jmonades}}]
    If $T$ is idempotent, then the Kleisli category $\Kl(T)$ is initial amongst $j$-reflective resolutions of $T$.
\end{corollary}

\begin{proof}
    $k_T \radj j v_T$ is initial amongst all resolutions of $T$. Hence, when $T$ is idempotent, so that $k_T \radj j v_T$ is $j$-reflective, it is furthermore initial amongst $j$-reflective resolutions.
\end{proof}

\subsection{Algebras and reflective resolutions}

To provide the promised analogue of \cref{forgetful-functor-is-ff} for relative monads, we give a characterisation of idempotent algebras in terms of morphisms of algebras.

\begin{lemma}
    \label{idempotence-via-morphisms}
    Let $(e', \rtimes')$ be a $T$-algebra. $(e', \rtimes')$ is idempotent if and only if, for every $T$-algebra $(e, \rtimes)$, every morphism from $e$ to $e'$ in $E$ is a $T$-algebra morphism.
\end{lemma}

\begin{proof}
    Suppose that $(e', \rtimes')$ is idempotent, and let $f \colon ja \to e$ be a morphism. We have $\epsilon f^\rtimes \eta_a = \epsilon f$ using the unit law for $e$. Thus, $\epsilon f^\rtimes = (\epsilon f^\rtimes \eta_a)^{\rtimes'} = (\epsilon f)^{\rtimes'}$ using idempotence of $(e', \rtimes')$, so that $\epsilon$ is a $T$-algebra morphism.

    Conversely, suppose that every morphism from the carrier of a $T$-algebra to $e'$ is a $T$-algebra morphism, and let $f \colon tx \to e'$ be a morphism. $ta$ is the carrier of the free $T$-algebra on $a \in A$ so that the $T$-algebra morphism law with respect to $\eta_a \colon ja \to ta$ gives that $(f \eta)^{\rtimes'} = f {\eta_a}^\dag = f$ using the unit law for the relative monad. Thus $(e', \rtimes')$ is idempotent.
\end{proof}

\begin{proposition}[{\cf{}~\cite[\S4]{diers1975jmonades}}]
    \label{Alg_idem-is-universal-j-reflective-resolution}
    Suppose that $T$ is idempotent. Then $\Alg\idem(T) \to E$ is $j$-reflective. Furthermore, it is terminal amongst reflective resolutions of $T$.
\end{proposition}

\begin{proof}
    Since $T$ is idempotent, free algebras are idempotent, so that $f_T \colon A \to \Alg(T)$ factors through $\Alg\idem(T)$. This forms a relative adjunction for the same reason as $\Alg(T)$, and it is trivially a resolution of $T$ (abstractly, this follows from \ffness{} of $\Alg\idem(T) \to \Alg(T)$~\cite[Remark~5.33]{arkor2024formal}). Furthermore, the forgetful functor $\Alg\idem(T) \to \Alg(T) \to E$ is full by \cref{idempotence-via-morphisms}, hence \ff{}. The proof of terminality goes through in essentially the same way as in the classical case, observing that, for a $j$-reflective resolution of $T$, the right $j$-adjoint $r \colon C \to E$ induces an idempotent $T$-algebra structure on $rc$ for each $c \in C$.
\end{proof}

\begin{corollary}
    \label{T-algebraically-idempotent-iff-Alg(T)-is-j-reflective}
    $T$ is algebraically idempotent if and only if the forgetful functor $u_T \colon \Alg(T) \to E$ is $j$-reflective.
\end{corollary}

\begin{proof}
    If $T$ is algebraically idempotent, then $\Alg\idem(T) \to \Alg(T)$ is an isomorphism, from which $j$-reflectivity follows from \cref{Alg_idem-is-universal-j-reflective-resolution}. Conversely, if $u_T$ is full, every $T$-algebra is idempotent by \cref{idempotence-via-morphisms}.
\end{proof}

For idempotent (non-relative) monads, every algebra is free, \ie{} the comparison functor $\b{Kl}(T) \to \Alg(T)$ is essentially surjective, hence an equivalence. This property is necessary to prove, for instance, that every reflective adjunction is monadic. In contrast, this phenomenon no longer holds for (even algebraically) idempotent relative monads.

\begin{example}[Algebraic idempotence $\centernot\implies$ every algebra is free]
    Let $E$ be the free category on the following graph.
    \[\begin{tikzcd}
    	j & t & e
    	\arrow["\eta", from=1-1, to=1-2]
    	\arrow["f", from=1-2, to=1-3]
    \end{tikzcd}\]
    The unique induced $(1 \to E)$-relative monad $T$ is idempotent, as is the unique algebra structure on $e$, hence $T$ is algebraically idempotent. The Kleisli category of $T$ is the trivial category. The category of algebras has two objects: $(t, \dag)$ and $(e, \rtimes)$, which are not isomorphic since $f$ is not invertible. Consequently the comparison functor $\b{Kl}(T) \to \Alg(T)$ is not an equivalence.
\end{example}

Consequently, for an algebraically idempotent relative monad $T$, we have that $v_T \colon \Kl(T) \to E$ and $u_T \colon \Alg(T) \to E$ are both $j$-reflective by \cref{T-idempotent-iff-Kl(T)-is-j-reflective,T-algebraically-idempotent-iff-Alg(T)-is-j-reflective}, but $v_T$ is not necessarily $j$-relatively monadic in the sense of \cite[Definition~4.1]{arkor2024relative}. Therefore, it is not generally true that $j$-reflections are $j$-relatively monadic.

\subsection{Algebraic idempotence from idempotence}

We conclude by establishing sufficient conditions for a relative monad to be algebraically idempotent as soon as it is idempotent. First, we present a method to deduce the algebraic idempotence of one relative monad from the algebraic idempotence of another (whose verification in practice is expected to be simpler). Recall from \cref{composing} that, given functors $j \colon A \to E$ and $j' \colon E \to I$, every $j'$-relative monad $T$ induces a $j'j$-relative monad $Tj$ by precomposing $j$. Furthermore, every $T$-algebra induces a $Tj$-algebra, inducing a functor $\Alg(T) \to \Alg(Tj)$ that commutes with the forgetful functors~\cite[Proposition~6.58]{arkor2024formal}.

\begin{definition}[{\cite[\S5.4]{arkor2022monadic}}]
    Let $j \colon A \to E$ and $j' \colon E \to I$ be functors. A $j'$-relative monad $T$ is \emph{$j$-ary} if the canonical functor $\Alg(T) \to \Alg(Tj)$ is an equivalence.
\end{definition}

Conceptually, a $j'$-relative monad $T$ is $j$-ary if it is presented by the data of a $j'j$-relative monad.

\begin{proposition}
    \label{j-ary-and-algebraic-idempotence}
    Let $j \colon A \to E$ and $j' \colon E \to I$ be functors and let $T$ be a $j'$-relative monad. If $T$ is $j$-ary, then $T$ is algebraically idempotent if and only if $Tj$ is algebraically idempotent.
\end{proposition}

\begin{proof}
    Under the assumption that $T$ is $j$-ary, we have a commutative triangle as follows.
    \[\begin{tikzcd}
    	{\Alg(T)} && {\Alg(Tj)} \\
    	& E
    	\arrow["\equiv", from=1-1, to=1-3]
    	\arrow["{u_T}"', from=1-1, to=2-2]
    	\arrow["{u_{Tj}}", from=1-3, to=2-2]
    \end{tikzcd}\]
    Consequently, each of the functors $u_T$ and $u_{Tj}$ is \ff{} if and only if the other is, which characterises algebraic idempotence by \cref{T-algebraically-idempotent-iff-Alg(T)-is-j-reflective}.
\end{proof}

Next, we observe that idempotence of a monad may in some situations be inferred from idempotence of the relative monad induced by precomposing a functor, giving a partial converse to \cref{composing}.

\begin{proposition}
    \label{idempotence-and-rank}
    Let $j \colon A \to E$ be a functor and let $T = (t, \mu, \eta)$ be a monad on $E$. If $t$ exhibits the left extension $j \lx (tj)$ of $tj \colon A \to E$ along $j$, then $T$ is idempotent if and only if $Tj$ is idempotent.
\end{proposition}

\begin{proof}
    Suppose that $Tj$ is idempotent. Naturality of $\eta$ implies the following diagram commutes for all $a \in A$.
    \[\begin{tikzcd}
    	tja & ttja \\
    	ja & tja
    	\arrow["{t\eta_{ja}}", from=1-1, to=1-2]
    	\arrow["{\eta_{ja}}", from=2-1, to=1-1]
    	\arrow["{\eta_{ja}}"', from=2-1, to=2-2]
    	\arrow["{\eta_{tja}}"', from=2-2, to=1-2]
    \end{tikzcd}\]
    Since $Tj$ is idempotent, precomposition with $\eta_{ja}$ exhibits a bijection of hom-sets, and so the above implies $t\eta_{ja} = \eta_{tja}$. By assumption, we have that precomposition by $j$ induces an isomorphism
    $[E, E](t, tt) \iso [E, E](j \lx (tj), tt) \iso [A, E](tj, ttj)$. Since $t\eta$ and $\eta t$ have the same whiskering with $j$, they are thus equal. Hence $T$ is idempotent. The converse holds by \cref{composing}.
\end{proof}

We deduce that, for relative monads with \emph{well-behaved} roots in the sense of \cite[Definition~4.1]{altenkirch2015monads}, which correspond to monads relative to free cocompletions~\cite[215]{szlachanyi2017tensor}, there is no distinction between idempotence and algebraic idempotence.

\begin{corollary}
    Let $\Phi$ be a class of small categories and let $A$ be a small category. Denote by $\phi_A \colon A \to \overline\Phi(A)$ the free cocompletion of $A$ under $\Phi$-colimits. A $\phi_A$-relative monad is idempotent if and only if it is algebraically idempotent.
\end{corollary}

\begin{proof}
    Let $T$ be a $\phi_A$-relative monad. By \cite[Theorem~8.3]{arkor2025bicategories}, $T$ extends via left extension along $\phi_A$ to a $\phi_A$-ary monad $(\phi_A \lx T)$ on $\overline\Phi(A)$. This monad thus satisfies the assumptions of \cref{idempotence-and-rank}, so that it is idempotent if and only if $T$ is idempotent. Furthermore, by \cref{idempotent-monad-is-algebraically-idempotent}, $(\phi_A \lx T)$ is idempotent if and only if it is algebraically idempotent. Finally, by \cref{j-ary-and-algebraic-idempotence}, $(\phi_A \lx T)$ is algebraically idempotent if and only if $T$ is algebraically idempotent.
\end{proof}

\printbibliography

\end{document}

%% file: core.tex
\newcommand{\ifarticle}[2]{
    \csname@ifclassloaded\endcsname{beamer}{#2}{#1}
}

\newcommand{\ifbook}[2]{
    \csname@ifclassloaded\endcsname{amsbook}{#1}{#2}
}

    \usepackage[T1]{fontenc} 
    \usepackage{xparse} 
    \usepackage{xspace} 
    \usepackage{etoolbox} 
    \usepackage{xpatch} 
    \usepackage{pgffor} 

    \usepackage{amsmath,amssymb,amsthm} 
    \allowdisplaybreaks[1]

    \usepackage{mathtools} 
    \usepackage{mathrsfs} 
    \usepackage{stmaryrd} 
    \usepackage{bbm} 
    \usepackage{quiver} 
    \usepackage{xcolor} 
    \usepackage{scalerel} 
    \usepackage{booktabs} 
    \usepackage{nameref} 
    \usepackage[british]{babel} 
    \usepackage{csquotes} 
    \usepackage[UKenglish]{isodate} 
    \usepackage{microtype} 
    \usepackage[splitrule]{footmisc} 

    \ifarticle{
        \PassOptionsToPackage{hyphens}{url}
        \usepackage[pdfusetitle,linktoc=all,colorlinks,citecolor=cyan,linkcolor=purple,urlcolor=blue]{hyperref} 
        \urlstyle{rm}

        \usepackage[nameinlink,noabbrev,capitalize]{cleveref} 
        \usepackage{footnotebackref} 

        \usepackage[shortlabels]{enumitem} 
        \setlist{topsep=2pt,itemsep=2pt,partopsep=2pt,parsep=2pt} 

        \setcounter{tocdepth}{1}
    }{}
    \usepackage[style=alphabetic,abbreviate=false,backref=true,backrefstyle=three,maxnames=9,minalphanames=3,maxalphanames=4]{biblatex} 
    \DeclareUnicodeCharacter{0301}{\TODO[Invalid symbol.]} 

    \DeclareFieldFormat*{citetitle}{\emph{#1}} 
    \DeclareCiteCommand{\citetitle}
        {\boolfalse{citetracker}%
        \boolfalse{pagetracker}%
        \usebibmacro{prenote}}
        {\ifciteindex
            {\indexfield{indextitle}}
            {}%
        \printtext[bibhyperref]{\printfield[citetitle]{labeltitle}}}
        {\multicitedelim}
        {\usebibmacro{postnote}}
    \DeclareCiteCommand*{\citeauthor}
        {\defcounter{maxnames}{99}%
        \defcounter{minnames}{99}%
        \defcounter{uniquename}{2}%
        \boolfalse{citetracker}%
        \boolfalse{pagetracker}%
        \usebibmacro{prenote}}
        {\ifciteindex{\indexnames{labelname}}{}%
        \printnames{labelname}}
        {\multicitedelim}
        {\usebibmacro{postnote}}

    \usepackage{calc} 
    \usepackage{tabularray} 
    \usepackage{ebproof} 
    \usepackage{forest} 
    \usepackage{adjustbox} 
    \usepackage{float} 
    \usepackage{stfloats}

    \makeatletter
    \ifarticle{
        \xpretocmd{\@adminfootnotes}{\let\@makefntext\BHFN@OldMakefntext}{}{}
        \renewcommand\@makefntext[1]{%
        \@ifundefined{@makefnmark}
            {}
            {%
            \renewcommand\@makefnmark{%
            \mbox{%
                \textsuperscript{%
                \normalfont
                \hyperref[\BackrefFootnoteTag]{\@thefnmark}%
                }%
            }\,%
            }%
            \BHFN@OldMakefntext{#1}%
        }%
        }

        \LetLtxMacro{\BHFN@Old@footnotemark}{\@footnotemark}
        \renewcommand*{\@footnotemark}{%
            \refstepcounter{BackrefHyperFootnoteCounter}%
            \xdef\BackrefFootnoteTag{bhfn:\theBackrefHyperFootnoteCounter}%
            \label{\BackrefFootnoteTag}%
            \BHFN@Old@footnotemark
        }

        \makeatletter
        \def\paragraph{\@startsection{paragraph}{4}%
          \z@\z@{-\fontdimen2\font}%
          {\normalfont\bfseries}}
        \makeatother
    }{}
    \makeatother

    \ifarticle{
        \theoremstyle{plain}
        \ifbook{
            \newtheorem{theorem}{Theorem}[chapter]
        }{
            \newtheorem{theorem}{Theorem}[section]
        }
        \newtheorem{proposition}[theorem]{Proposition}
        \newtheorem{lemma}[theorem]{Lemma}
        \newtheorem{corollary}[theorem]{Corollary}

        \newtheorem*{theorem*}{Theorem}
        \newtheorem*{corollary*}{Corollary}

        \theoremstyle{definition}
        \newtheorem{definition}[theorem]{Definition}
        
        \newtheorem{example}[theorem]{Example}

        \newtheorem{remark}[theorem]{Remark}
        
        \Crefname{axiom}{Axiom}{Axioms}

        \newenvironment{sketch}{\proof}{\endproof}

        \Crefname{theoremenumi}{Theorem}{Theorems}
        \AtBeginEnvironment{theorem}{%
            \crefalias{enumi}{theoremenumi}%
            \setlist[enumerate,1]{
                ref={\csname thetheorem\endcsname.(\arabic*)}
            }%
            \crefalias{enumii}{theoremenumi}%
            \setlist[enumerate,2]{
                ref={\thetheorem.(\arabic*).(\alph*)}
            }%
        }

        \Crefname{propositionenumi}{Proposition}{Propositions}
        \AtBeginEnvironment{proposition}{%
            \crefalias{enumi}{propositionenumi}%
            \setlist[enumerate,1]{
                ref={\csname theproposition\endcsname.(\arabic*)}
            }%
            \crefalias{enumii}{propositionenumi}%
            \setlist[enumerate,2]{
                ref={\theproposition.(\arabic*).(\alph*)}
            }%
        }

        \Crefname{lemmaenumi}{Lemma}{Lemmas}
        \AtBeginEnvironment{lemma}{%
            \crefalias{enumi}{lemmaenumi}%
            \setlist[enumerate,1]{
                ref={\csname thelemma\endcsname.(\arabic*)}
            }%
            \crefalias{enumii}{lemmaenumi}%
            \setlist[enumerate,2]{
                ref={\thelemma.(\arabic*).(\alph*)}
            }%
        }

        \Crefname{corollaryenumi}{Corollary}{Corollaries}
        \AtBeginEnvironment{corollary}{%
            \crefalias{enumi}{corollaryenumi}%
            \setlist[enumerate,1]{
                ref={\csname thecorollary\endcsname.(\arabic*)}
            }%
            \crefalias{enumii}{corollaryenumi}%
            \setlist[enumerate,2]{
                ref={\thecorollary.(\arabic*).(\alph*)}
            }%
        }

        \Crefname{definitionenumi}{Definition}{Definitions}
        \AtBeginEnvironment{definition}{%
            \crefalias{enumi}{definitionenumi}%
            \setlist[enumerate,1]{
                ref={\csname thedefinition\endcsname.(\arabic*)}
            }%
            \crefalias{enumii}{definitionenumi}%
            \setlist[enumerate,2]{
                ref={\thedefinition.(\arabic*).(\alph*)}
            }%
        }

        \Crefname{exampleenumi}{Example}{Examples}
        \AtBeginEnvironment{example}{%
            \crefalias{enumi}{exampleenumi}%
            \setlist[enumerate,1]{
                ref={\csname theexample\endcsname.(\arabic*)}
            }%
            \crefalias{enumii}{exampleenumi}%
            \setlist[enumerate,2]{
                ref={\theexample.(\arabic*).(\alph*)}
            }%
        }

        \Crefname{axiomenumi}{Axiom}{Axioms}
        \AtBeginEnvironment{axiom}{%
            \crefalias{enumi}{axiomenumi}%
            \setlist[enumerate,1]{
                ref={\csname theaxiom\endcsname.(\arabic*)}
            }%
            \crefalias{axiomii}{axiomenumi}%
            \setlist[enumerate,2]{
                ref={\theaxiom.(\arabic*).(\alph*)}
            }%
        }

        \foreach \env in {definition,remark,example,notation,axiom}{
        \AtBeginEnvironment{\env}{%
          \pushQED{\qed}%
        }
        \AtEndEnvironment{\env}{\popQED\endexample}
        }
    }{}

    \ExplSyntaxOn
    \NewDocumentCommand{\mathcommand}{mO{0}m}
     {
      \exp_args:Nc \NewCommandCopy {khue_kept_\cs_to_str:N #1} { #1 }
      \exp_args:Nc \newcommand {khue_new_\cs_to_str:N #1}[#2]{#3}
      \DeclareDocumentCommand {#1} {}
       {
        \mode_if_math:TF
         {
          \use:c {khue_new_\cs_to_str:N #1}
         }
         {
          \use:c {khue_kept_\cs_to_str:N #1}
         }
       }
     }
    \ExplSyntaxOff

    \newsavebox\tikzcdbox

    \mathcommand{\h}{\textup{-}}

    \newcommand{\tx}{\mathrm}
    \mathcommand{\b}{\mathbf}
    
    \mathcommand{\bb}{\mathbb}
    \DeclareMathAlphabet{\bbn}{U}{bbold}{m}{n}
    
    \mathcommand{\sf}{\mathsf}
    \mathcommand{\u}{\underline}
    \mathcommand{\o}{\overline}

    \newcommand{\TODO}[1][TODO]{\textcolor{orange}{\textup{#1}}\xspace}

    \newcommand{\flip}[1]{\text{\rotatebox[origin=c]{-180}{$#1$}}}
    
    \newcommand{\datetoday}{\date{\cleanlookdateon\today}}

    \newcommand{\defeq}{\mathrel{:=}}
    \mathcommand{\d}{\mathbin{;}}
    \mathcommand{\c}{\circ}
    \newcommand{\ph}[1][]{{({-}_{#1})}}
    
    \newcommand{\iso}{\cong}
    \renewcommand{\equiv}{\simeq}

    \newcommand{\tto}{\Rightarrow}

    \makeatletter
    \def\slashedarrowfill@#1#2#3#4#5{%
    $\m@th\thickmuskip0mu\medmuskip\thickmuskip\thinmuskip\thickmuskip
    \relax#5#1\mkern-7mu%
    \cleaders\hbox{$#5\mkern-2mu#2\mkern-2mu$}\hfill
    \mathclap{#3}\mathclap{#2}%
    \cleaders\hbox{$#5\mkern-2mu#2\mkern-2mu$}\hfill
    \mkern-7mu#4$%
    }
    \def\rightslashedarrowfill@{%
    \slashedarrowfill@\relbar\relbar\mapstochar\rightarrow}
    \newcommand\xslashedrightarrow[2][]{%
    \ext@arrow 0055{\rightslashedarrowfill@}{#1}{#2}}
    \def\leftslashedarrowfill@{%
    \slashedarrowfill@\leftarrow\relbar\mapsfromchar\relbar}
    \newcommand\xslashedleftarrow[2][]{%
    \ext@arrow 0055{\leftslashedarrowfill@}{#1}{#2}}
    \makeatother

    \newcommand{\op}{{}^\tx{op}}

    \newcommand{\lx}{\mathbin{\rhd}}
    
    \newcommand{\adj}{\dashv}
    
    \newcommand{\radj}[1]{\mathrel{\adj_{#1}}}

    \DeclareFontFamily{U}{min}{}
    \DeclareFontShape{U}{min}{m}{n}{<-> udmj30}{}

    \mathcommand{\comma}{\downarrow}

    \newcommand{\copi}{\flip\pi}
    
    \newsavebox{\whitecircstar}\sbox{\whitecircstar}{\kern.075em\tikz{\node[draw, circle,line width=.36pt, inner sep=0]{$*$};}\kern.075em}
    
    \newsavebox{\blackcircstar}\sbox{\blackcircstar}{\kern.075em\tikz{\node[fill, circle, line width=.36pt, inner sep=0, text=white]{$*$};}\kern.075em}

    \makeatletter
    \def\widebreve{\mathpalette\wide@breve}
    \def\wide@breve#1#2{\sbox\z@{$#1#2$}%
         \mathop{\vbox{\m@th\ialign{##\crcr
    \kern0.08em\brevefill#1{0.8\wd\z@}\crcr\noalign{\nointerlineskip}%
                        $\hss#1#2\hss$\crcr}}}\limits}
    \def\brevefill#1#2{$\m@th\sbox\tw@{$#1($}%
      \hss\resizebox{#2}{\wd\tw@}{\rotatebox[origin=c]{90}{\upshape(}}\hss$}
    \makeatother

    \NewDocumentCommand{\jrule}{om}{%
        \IfNoValueTF{#1}
            {\textsc{#2}}
            {$#1$-\textsc{#2}}%
    }

    \newcommand{\Set}{{\b{Set}}}
    
    \newcommand{\Alg}{\b{Alg}}
    
    \newcommand{\ff}{fully faithful}
    
    \newcommand{\ffness}{full faithfulness}

    \newcommand{\ioo}{identity-on-objects}

    \newcommand{\ie}{i.e.\@\xspace}
    
    \newcommand{\cf}{cf.\@\xspace}

    \NewDocumentCommand{\etc}{t.}{etc.\@\xspace}
    \NewDocumentCommand{\ibid}{t.}{ibid.\@\xspace}
    \NewDocumentCommand{\loccit}{t.}{loc.\ cit.\@\xspace}

\makeatletter

\define@key{beamerframe}{c}[true]{
    \beamer@frametopskip=0pt plus 1fill\relax%
    \beamer@framebottomskip=0pt plus 1fill\relax%
}

\patchcmd{\beamer@sectionintoc}{\vfill}{\vskip\itemsep}{}{}

\makeatother

\ifarticle{}{

  \colorlet{colour-bg}{black!85} 
  \definecolor{colour-primary}{HTML}{cc80ff} 
  \colorlet{colour-text}{black!10} 
  \colorlet{colour-subtle}{black!40} 
  \colorlet{colour-block-bg}{black!80} 
  \definecolor{colour-warning-bg}{HTML}{ffea80} 
  \definecolor{colour-warning-primary}{HTML}{e08152} 

  \hypersetup{linktoc=all,colorlinks, citecolor={colour-primary}, linkcolor={colour-primary}, urlcolor={colour-primary}} 
  \urlstyle{sf}

  \usepackage{changepage} 

  \usepackage{ragged2e} 

  \setbeamertemplate{section in toc}[sections numbered]

  \apptocmd{\frame}{}{\justifying}{}

  \usefonttheme[onlymath]{serif}

  \beamertemplatenavigationsymbolsempty
  \setbeamertemplate{footline}{
    \hfill%
    \usebeamercolor[fg]{page number in head/foot}%
    \usebeamerfont{page number in head/foot}%
    \setbeamertemplate{page number in head/foot}[pagenumber]%
    \usebeamertemplate*{page number in head/foot}\kern.2cm\vskip.2cm%
  }

  \setbeamertemplate{frametitle}[default][center]
  \addtobeamertemplate{frametitle}{\vspace*{1ex}}{\vspace*{1ex}}

  \setbeamertemplate{itemize item}{$\bullet$}

  \setbeamertemplate{theorems}[numbered]
  \newtheorem{proposition}[theorem]{\translate{Proposition}}

  \addtobeamertemplate{block begin}
  {}
  {\vspace{0ex} 
  \begin{adjustwidth}{1em}{1em} 
  \tikzcdset{background color=colour-block-bg}
  }
  \addtobeamertemplate{block end}{\end{adjustwidth}%
  \vspace{1ex}} 
  {}
  \renewenvironment<>{block}[1]{%
      \begin{actionenv}#2%
        \def\insertblocktitle{\begin{adjustwidth}{.8em}{.8em}#1\end{adjustwidth}}%
        \par%
        \usebeamertemplate{block begin}}
      {\par%
        \usebeamertemplate{block end}%
      \end{actionenv}}

  \addtobeamertemplate{block example begin}
  {}
  {\vspace{0ex} 
  \begin{adjustwidth}{1em}{1em} 
  \tikzcdset{background color=colour-block-bg}
  }
  \addtobeamertemplate{block example end}{\end{adjustwidth}%
  \vspace{1ex}} 
  {}
  \renewenvironment<>{exampleblock}[1]{%
      \begin{actionenv}#2%
          \def\insertblocktitle{\begin{adjustwidth}{.8em}{.8em}#1\end{adjustwidth}}%
          \par%
          \only<presentation>{
            \setbeamercolor{local structure}{parent=example text}}%
          \usebeamertemplate{block example begin}}
        {\par%
          \usebeamertemplate{block example end}%
        \end{actionenv}}

    \setbeamercolor{background canvas}{bg=colour-bg}
    \tikzcdset{background color=colour-bg}

    \setbeamercolor{block title}{fg=colour-primary,bg=colour-block-bg}
    \setbeamercolor{block title example}{fg=colour-primary,bg=colour-block-bg}
    \setbeamercolor{block body}{bg=colour-block-bg}
    \setbeamercolor{block body example}{bg=colour-block-bg}

    \setbeamercolor{title}{fg=colour-primary}
    \setbeamercolor{frametitle}{fg=colour-primary}
    \setbeamercolor{alerted text}{fg=colour-primary}

    \setbeamercolor{normal text}{fg=colour-text}
    \setbeamercolor{item}{fg=colour-text}
    \setbeamercolor{section in toc}{fg=colour-text}

    \setbeamercolor{page number in head/foot}{fg=colour-subtle}

    \setbeamercolor*{bibliography entry author}{fg=colour-text}
    \setbeamercolor*{bibliography entry note}{fg=colour-text}

}